\documentclass[article,12pt]{amsart}
\usepackage{lineno,hyperref,enumitem}
\usepackage{amsmath,amssymb,dsfont}

\newtheorem{thm}{Theorem}[section]
\newtheorem{defn}[thm]{Definition}
\newtheorem{prop}[thm]{Proposition}

\newtheorem{lem}[thm]{Lemma}
\newtheorem{rem}[thm]{Remark}

\newtheorem{ex}[thm]{Example}

\begin{document}

\title[Scaling Set and Generalized Scaling Set]{$p$-Adic Scaling Set and Generalized Scaling Set}

\author[D. Haldar]{Debasis Haldar}

\address{Department of Mathematics\\ NIT Rourkela\\ Rourkela 769008\\ India}

\email{debamath.haldar@gmail.com}

\subjclass[2010]{42C40, 11F85}

\keywords{$p$-adic, scaling set, multiwavelet set, generalized scaling set}

\begin{abstract}
The main goal of this paper is to develop the MRA theory along with wavelet theory in $L^{2}(\mathbb{Q}_p)$. Generalized scaling sets are important in wavelet theory because it determine multiwavelet sets. Although the theory of scaling set and generalized scaling set on $\mathbb{R}$ and many other local field of positive characteristic are available but not on $\mathbb{Q}_p$. This article contains discussion of some necessary conditions of scaling set and characterize generalized scaling set with examples.
\end{abstract}

\maketitle



\section{Introduction}
Wavelet analysis is consider as one of the pillar of modern harmonic analysis since last three decades. It is integral part of digital signal processing and image processing. Moreover, it also has wide applications in statistical data analysis \cite{ABS00}, nonlinear dynamics \cite{AWM98,GPT97}, spectral analysis \cite{K02}, quantum field theory \cite{AK13}, pseudo-differential operator theory \cite{KK05,KS09,K04}, to name a few. As wavelets are localised in time and frequency, it is useful where classical Fourier analysis is not effective. Mathematically, multiwavelets are some functions whose dilates and translates give some kind of basis, preferably an orthonormal basis.\\

First example of one dimensional $p$-adic multiwavelet was given by Kozyrev \cite{K02}. It is haar type. In that paper, he also established a relationship between wavelet analysis and spectral analysis by showing those multiwavelets are eigen functions of the Vladimirov operator. Later, Khrennikov and Shelkovich \cite{KS08,KS09} extended the example of $p$-adic multiwavelet given by Kozyrev and studied its connection with spectral analysis and cauchy problem in PDE. Their multiwavelet is non-haar type.\\

Albeverio et al. \cite{AKS06} gave an example of $n$-dimensional $p$-adic multiwavelet by direct multiplying Kozyrev's multiwavelet. Moreover, Kozyrev's multiwavelet was extended to the ultrametric spaces \cite{KK205,KK05,K04}.\\

J. J. Benedetto and R. L. Benedetto \cite{BB04,B04} suggested a method for multiwavelet based on multiwavelet set on the locally compact abelian groups having compact open subgroup. As $\mathbb{Q}_p$ does not have any discrete subgroup, they expressed their doubt on the development of $p$-adic multiwavelet theory through MRA. However, Kozyrev's multiwavelet can be constructed in the Benedettos' framework.\\

Shelkovich and Skopina \cite{SS9} initiated development of $p$-adic MRA and wavelet theory. They have given example of scaling set and multiwavelet on $L^{2}(\mathbb{Q}_2)$. Khrennikov et al. \cite{KSS09} developed $p$-adic MRA theory by describing a wide class of $p$-adic refinement equations generating MRA. Also, Albeverio et al. \cite{AES10} parallelly developed $p$-adic MRA theory. They showed 1-periodic test functions only be considered as scaling functions and those generate haar MRA. Evdokimov and Skopina \cite{ES15} showed that any band-limited (periodic) multiwavelet is nothing but a modification of the Kozyrev's multiwavelet.\\

The organization of the present paper is as follows. In Section 2, we discuss preliminary results related to $p$-adic number field $\mathbb{Q}_p$ and $\mathbb{C}$-valued function defined on it. Section 3 deals with scaling set. Finally, characterization and examples of generalized scaling set are given in Section 4.

\section{Preliminaries}

Let us make here a brief review of $p$-adic analysis. We refer \cite{RV99,T75,VVZ94} for details and proof of the results in this section. $\mathbb{Q}$ and $\mathbb{C}$ denote the set of rational numbers and the set of complex numbers respectively. For a prime $p$, the $p$-adic norm $|\cdot|_{p}$ on $\mathbb{Q}$ is defined as follows 
\begin{center}
	$|x|_{p}= \begin{cases}
		0 & \text{ if } x=0, \\
		p^{- \gamma} & \text{ if } x=p^{\gamma}\frac{m}{n} \neq 0 ,~~ p \not| ~ mn.
	\end{cases}$
\end{center}

This norm is non-Archimedean i.e. satisfies strong triangle inequality $\lvert x+y \rvert_{p} \leq \text{max} \{\lvert x \rvert_{p}, \lvert y \rvert _{p} \}$. Here equality holds if and only if $|x|_{p} \neq |y|_{p}$. The field $\mathbb{Q}_p$ of $p$-adic numbers is defined as the completion of $\mathbb{Q}$ with respect to the norm $|\cdot|_{p}$. Every non-zero $p$-adic number has following formal Laurent series expansion in $p$ as follows
$$x= \sum \limits_{j=\gamma}^{\infty} x_{j}p^j,$$
where $x_{j} \in \{0,1,\ldots, p-1\}$ with $x_{\gamma} \neq 0$ and $\gamma \in \mathbb{Z}$. The fractional part $\{x\}_{p}$ of $x$ is defined by $\sum \limits_{j=\gamma}^{-1}x_{j}p^j$. Note that  $\{x\}_{p} =0$ if and only if $\gamma \geq 0$. Moreover, $\{0\}_{p}:=0$. Define $J_{p,m}:= \{\frac{s_{-m}}{p^m}+ \ldots +\frac{s_{-1}}{p} : s_{-j} = 0, 1,\ldots, p-1;~ j = 1,2,\ldots, m;~ s_{-m} \neq 0 \}$ for $m \in \mathbb{N}$. The set of $p$-adic integers $\mathbb{Z}_{p}$ and the set of $p$-adic fractional numbers $I_{p}$ are given by
\begin{center}
	$\mathbb{Z}_{p} = \{ x \in \mathbb{Q}_{p} : \{x\}_{p}=0 \}$ and $I_{p} = \{ x \in \mathbb{Q}_{p}: \{x\}_{p} = x \}$.
\end{center}
We can also write $\mathbb{Z}_{p} = \{x \in \mathbb{Q}_{p} : |x|_{p} \leq 1 \}$. The translates of $\mathbb{Z}_{p}$ by the elements of $I_{p}$ are mutually disjoint and their union equals $\mathbb{Q}_{p}$. Therefore $I_{p}$ can be considered as a set  of translations for $\mathbb{Q}_{p}$. Note that although $I_{p}$ is a discrete subset of $\mathbb{Q}_{p}$ but does not form a group. Moreover, any choice of representatives in $\mathbb{Q}_{p} / \mathbb{Z}_p$ is discrete but not a group as there is only trivial discrete subgroup of $\mathbb{Q}_{p}$. It is become the biggest constrain for development of $p$-adic MRA theory as well as wavelet theory compaired to $\mathbb{R}$, Vilenkin group, local field of positive characteristic etc. The additive character $\chi_{p}$ on the field $\mathbb{Q}_{p}$  is defined by
$$\chi_{p}(x)= e^{2\pi i \{x\}_{p}},~~~ x \in \mathbb{Q}_{p}.$$

\noindent A ball centered at $c \in \mathbb{Q}_{p}$ with radius $p^\gamma$, where $\gamma \in \mathbb{Z}$, denoted as $B_{\gamma}(c)$ and defined by
$$B_{\gamma}(c) = \{x \in \mathbb{Q}_{p}: \lvert x-c \rvert_{p} \leq p^{\gamma} \}.$$

The ball $B_{\gamma}(c)$ is compact, open and every point of the ball is also a center of the ball. Further any two balls are either disjoint or one is contained in the other. The field $\mathbb{Q}_{p}$ is locally compact, totally disconnected and has no isolated points. There exists Haar measure $dx$ on $\mathbb{Q}_p$ which is positive and invariant under translations, i.e., $d(x+a) = dx$, for all $a \in \mathbb{Q}_{p}$. It is normalized by the equality $\int \limits_{\mathbb{Z}_{p} }dx =1$. Moreover,
$$d(ax)= |a|_{p}~ dx,~~~  a \in \mathbb{Q}_{p} \setminus \{0 \}.$$

The Hilbert space of all complex-valued functions on $\mathbb{Q}_{p}$, square integrable with respect to the measure $dx$, is denoted by $ L^{2}(\mathbb{Q}_{p}) $. The inner product in this space is given by
\begin{center}
	$\langle f,g \rangle = \int \limits_{\mathbb{Q}_{p}} f(x) \overline{g(x)}dx,$ where $f, g \in L^{2}(\mathbb{Q}_{p}).$
\end{center}

A complex valued function defined on $\mathbb{Q}_{p}$ is said to be locally constant if for any $x \in \mathbb{Q}_{p}$ there exists an  $l(x) \in \mathbb{Z}$ such that $f(x+y)=f(x), ~ y \in B_{l(x)}(0)$. The linear space of locally-constant and compactly supported functions (also called the test functions) is denoted by $\mathcal{D}$. The set of all complex valued functions defined on $\mathbb{Q}_{p}$ which are locally constant, $p^M$-periodic and supported on $B_{N}(0)$ is denoted by $\mathcal{D}_{N}^{M}$. Note that $f \in \mathcal{D}_{N}^{M}$ is equivalent to say that $f$ is $p^M$-periodic and Fourier transform of $f$ is supported on $B_{M}(0)$. 
The Fourier transform of $f \in L^{1}(\mathbb{Q}_p) \cap L^{2}(\mathbb{Q}_p)$ is defined as
\begin{eqnarray}
	\hat{f}(\xi) := \mathcal{F}[f](\xi)= \int \limits_{\mathbb{Q}_p}\chi_{p}(\xi x)f(x) ~ dx \nonumber
\end{eqnarray}
where  $\xi \in \mathbb{Q}_{p}$. This Fourier transform can be extended to $L^{2}(\mathbb{Q}_{p})$ and we have the Parseval theorem \cite{T75} as follows
$$\langle f,g \rangle = \langle \hat{f}, \hat{g} \rangle, \hspace{1 cm} \forall ~ f, g \in L^{2}(\mathbb{Q}_{p}).$$

The inverse Fourier transform of $f \in L^{2}(\mathbb{Q}_{p})$ is defined as

$$\check{f}(\xi) := \int \limits_{\mathbb{Q}_p}\chi_{p}(-\xi x)f(x) ~ dx.$$

It is easy to see that $\hat{\check{f}}=f$. If $f \in L^{2}(\mathbb{Q}_{p})$ and $a, b \in \mathbb{Q}_{p}$ with $a \neq 0$, then Fourier transform of dilate along with translate of $f$ is

$$\mathcal{F}[f(a \cdot +b)](\xi)= |a|^{-1}_{p} \chi_{p}\bigg(- \frac{b}{a}\xi\bigg) \mathcal{F}[f]\bigg(\frac{\xi}{a}\bigg).$$

For example, fourier tansform of the characteristic function of the set of $p$-adic integers is itself, i.e., $\hat{\mathds{1}}_{\mathbb{Z}_{p}} =  \mathds{1}_{\mathbb{Z}_{p}}$ and moreover, $\mathds{1}_{\mathbb{Z}_{p}}$ is 1-periodic.

\begin{lem}\label{zerolemma}
	If $f \in L^{2}(\mathbb{Q}_{p})$, then $\lim \limits_{n \to \infty} |f(p^{-n}x)|=0$ for $x \in \mathbb{Q}_{p}$ a.e.
\end{lem}

\begin{proof}
Let $f \in L^{2}(\mathbb{Q}_{p})$. By applying monotone convergence theorem we have,
\begin{eqnarray*}
	\int \limits_{\mathbb{Q}_p} \sum_{n \in \mathbb{N}} |f(p^{-n}x)|^{2}dx &=& \sum_{n \in \mathbb{N}} \int \limits_{\mathbb{Q}_p} |f(p^{-n}x)|^{2}dx \\
	&=& \sum_{n \in \mathbb{N}}p^{-n} \int \limits_{\mathbb{Q}_p} |f(x)|^{2}dx \\
	&=& \frac{1}{p-1} \|f\|^{2} <\infty
\end{eqnarray*}

It follows that, $\sum_{n \in \mathbb{N}} |f(p^{-n}x)|^{2}dx <\infty$ for $x \in \mathbb{Q}_{p}$ a.e.\\ Thus, $\lim \limits_{n \to \infty} |f(p^{-n}x)|=0$ for $x \in \mathbb{Q}_{p}$ a.e.
\end{proof}

\section{Scaling Set}
Shelkovich and Skopina \cite{SS9} gave following real analogous definition of $p$-adic multiresolution analysis (MRA in sequel).

\begin{defn}(Multiresolution Analysis) \label{defMRA}
	A collection of closed subspaces $V_{j} \subset L^2({\mathbb{Q}}_{p})$, where $j \in \mathbb{Z}$, is called a multiresolution analysis if following conditions hold
	\begin{enumerate}[label=(\roman*)]
		\item $V_{j} \subset V_{j+1}$ for all $j \in \mathbb{Z}$,
		\item $\overline{\bigcup\limits_{j \in \mathbb{Z}}V_{j}}=L^2({\mathbb{Q}}_{p})$,
		\item $\bigcap\limits_{j \in \mathbb{Z}}V_{j} =\{0\}$,
		\item $f(\cdot) \in V_{j} \iff f(p^{-1}\cdot) \in V_{j+1}$ for all $j \in \mathbb{Z}$,
		\item there exists $\phi \in V_{0}$ such that the system $\{\phi(\cdot -a): a \in I_{p}\}$ is an orthonormal basis for $V_{0}$.
	\end{enumerate}
\end{defn}

\noindent This $\phi$ is  called  a  scaling  function of the given MRA. We will also say that $\phi$ generates the MRA. Axioms $(iv)$ and $(v)$ of Definition \ref{defMRA} immediately refer that $\{p^{\frac{j}{2}}\phi(p^{-j}\cdot -a): a \in I_{p}\}$ forms an orthonormal basis for $V_{j}$. As $\phi \in V_{0} \subset V_{1}$ and $\{p^{\frac{1}{2}}\phi(p^{-1}\cdot -a): a \in I_{p}\}$ is an orthonormal basis for $V_{1}$, $\phi$ can be written as 
\begin{equation}
	\phi(x)= \sum \limits_{a \in I_p} \alpha_{a} \phi(p^{-1}x -a), \hspace{.5cm} \text{where } \alpha_{a}=p \int\limits_{\mathbb{Q}_p} \phi(x)\overline{\phi(p^{-1}x-a)} dx.
\end{equation}
In order to construct an MRA, we start with a function $\phi$ whose $I_p$-translates be an orthonormal system and set 
\begin{equation}\label{vj}
	V_{j}= \overline{span} \{\phi(p^{-j} \cdot -a) : a \in I_{p}\}, ~~~~ j \in \mathbb{Z}.
\end{equation}
Now axioms $(iv)$ and $(v)$ clearly hold here. In order to hold $(i)$, we also need $\phi$ such that $\phi(\cdot - b) \in V_{1}, ~~ \forall ~b \in I_p$, i.e., 
\begin{equation} \label{d1}
	\phi(\cdot -b) = \sum_{a \in I_p} \alpha_{a,b} \phi(p^{-1} \cdot -a), ~~~~~~~~ \forall ~b \in I_p.
\end{equation}
But, in general, $\frac{b}{p}+a \notin I_p$ for $a,b \in I_p$, therefore $\phi(\cdot - b) \notin V_1, ~~ \forall~b \in I_p$. Nevertheless, there are $\phi$'s which satisfy equation (\ref{d1}) and we will discuss here.

\begin{defn} (Scaling Set)
	A measurable set $S \subset \mathbb{Q}_{p}$ is said to be a scaling set if $\check{\mathds{1}}_S$ is a scaling function of an MRA in $L^2({\mathbb{Q}}_{p})$.
\end{defn}

\begin{ex}
	$\mathbb{Z}_p$ is a scaling set.
\end{ex}

\begin{lem}
	Measure of a scaling set is 1.
\end{lem}

\begin{proof}
	Let $S\subset \mathbb{Q}_{p}$ be a scaling set. As dilation is an isomorphism,\\ $\{p^{\frac{j}{2}}\check{\mathds{1}}_S (p^{-j}\cdot - a) : a \in I_p \}$ is an orthonormal basis for the space $V_j$. Now
	$$1= \langle p^{\frac{j}{2}}\check{\mathds{1}}_S (p^{-j}\cdot), p^{\frac{j}{2}} \check{\mathds{1}}_S (p^{-j}\cdot) \rangle = p^{-j} \langle \mathds{1}_{p^{-j}S}, \mathds{1}_{p^{-j}S} \rangle = p^{-j} \mu (p^{-j}S).$$
	
\noindent Therefore $\mu (p^{-j}S)= p^{j}$.
\end{proof}

\begin{prop}
	For a scaling set $S$, $\lim\limits_{n \to \infty} \mathds{1}_{p^{n}S}(\xi)=0$ for $\xi \in \mathbb{Q}_{p}$ a.e.
\end{prop}

\begin{proof}
By Plancherel Theorem, we know that, $\mathds{1}_{S} \in L^2({\mathbb{Q}}_{p})$. Now by Lemma \ref{zerolemma}, for $\xi \in \mathbb{Q}_{p}$ a.e.

$$\lim\limits_{n \to \infty} \mathds{1}_{S}(p^{-n}\xi)=0 \Rightarrow \lim\limits_{n \to \infty} \mathds{1}_{p^{n}S}(\xi)=0$$ 	
\end{proof}

Haldar and Singh \cite{HS19} defined following multiwavelet set analogous to $\mathbb{R}$.

\begin{defn}(Multiwavelet Set)
	A measurable set $W \subset \mathbb{Q}_{p}$ is said to be a multiwavelet set of order $L$ if $W = \bigcup\limits^{L}_{n=1}W_{n}$ for measurable sets $W_{n}$; $n=1,\ldots,L$ and $\{ \psi_{1}, \ldots, \psi_{L} \}$ is a multiwavelet for $L^{2}(\mathbb{Q}_{p})$ where $\hat{\psi}_{n}= \mathds{1}_{W_{n}}$.\\
	When $L=1$, $W$ is simply said to be a wavelet set.
\end{defn}

\begin{lem}
	For a MRA, scaling set and wavelet set are disjoint a.e.
\end{lem}

\begin{proof}
Let $W$ and $S$ be wavelet set and scaling set. As $\check{\mathds{1}}_W \perp V_0$ 

 $$0= \langle \check{\mathds{1}}_W, \check{\mathds{1}}_S (\cdot -a) \rangle = \int \limits_{S \cap W} \chi_{p}(a\xi) d\xi$$
 
 If $a=0$, then $\mu(S \cap W) =0$.
\end{proof}	
Note that instead of wavelet set, if it is a multiwavelet set then\\ $\mu((\bigcup \limits_{n=1}^{L} W_n) \cap S) =0$ i.e. 	$\mu(\bigcup \limits_{n=1}^{L}( W_n \cap S)) =0.$

\begin{defn}(Refinable Function)
	A function $f \in L^{2}(\mathbb{Q}_p)$ is said to be a refinable function if it satisfy following refinement equation $$f(x)= \sum \limits_{a \in I_p} r_{a} f\left(\frac{x}{p} - a\right), \hspace{1cm} r_{a} \in \mathbb{C}.$$
\end{defn}

Note that every scaling function is refinable function but converse may not be true.

\begin{lem} \label{inc}
	If $S \subset \mathbb{Q}_p$ is a scaling set, then $S \subset p^{-1}S$. 
\end{lem}

\begin{proof}
Let $\phi$ be the scaling function corresponding to $S$. We know scaling function satisfy following refinement equation 
\begin{equation*}
\phi(x)= \sum \limits_{a \in I_p} \alpha_{a} \phi(p^{-1}x -a), \hspace{1cm} \text{where } \alpha_{a} \in \mathbb{C}.
\end{equation*}

Taking Fourier transform in both side, we have 
\begin{equation} \label{Em}
	\hat{\phi}(\xi)= m_{0}(p\xi) \hat{\phi}(p\xi), \text{   where } m_{0}(\xi):= \sum \limits_{a \in I_p} \frac{\alpha_{a}}{p} \chi_{p}{(a\xi)}
\end{equation}
 $\Rightarrow \mathds{1}_{S}(\xi)= \mathds{1}_{p^{-1}S}(\xi) m_{0}(p\xi)$

This completes the proof. 	
\end{proof}

$m_0$ is said to be scaling filter. It is to be noted that $m_0$ is trigonometric function satisfying $m_{0}(0)= 1$ whenever $\hat{\phi}(0) \neq 0.$

\begin{prop}
Let $S$ be a scaling set and $m_0$ be corresponding filter defined by equation (\ref{Em}). Then $$m_{0}(p\xi) = \begin{cases}
	1 & \text{ if } \xi \in S \\
	0 & \text{ if } \xi \in p^{-1}S \setminus S.
\end{cases}$$ 
\end{prop}

\begin{proof}
As $S$ is the scaling set, $$\mathds{1}_{S}(\xi)=\mathds{1}_{p^{-1}S}(\xi) m_{0}(p\xi).$$

\noindent If $\xi \in S$ then, by Lemma \ref{inc}, $m_{0}(p\xi)=1$ .\\
If $\xi \notin S$ but $\xi \in p^{-1}S$ then $m_{0}(p\xi)=0$.
\end{proof}

\begin{prop}
	Let $S \subset B_{M}(0)$ be scaling set correspond to $\phi$ defined by $\hat{\phi}(\xi)= R \prod \limits_{j=0}^{\infty}m\left(\frac{\xi}{p^{N-j}}\right)$, where $m$ is a trigonometric polynomial satisfy $m(0)=1$ and $R$ is a real no. Then at most $\frac{\text{deg}~m}{p-1}$ many integers are in $p^{M}S \cap [0,p^{M+N}-1]$. 
\end{prop}

\begin{proof}
	It easily follows from Lemma 1 of \cite{AES10}.
\end{proof}

\begin{prop}
Let $S$ be scaling set containg $0$ corresponding $\phi \in \mathfrak{D}_{N}^{M}$, where $M,N \geq 0$. If $\forall b ~ \in \{x \in I_p : \left|x \right|_{p} \leq p^N \}$, $\phi(\cdot -b) \in V_{1}$ as defined by equation (\ref{vj}), then there are at most $p^N$ many imterger $n$ such that $n \in p^{M}S \cap [0,p^{M+N}-1]$.	
\end{prop}	
	
\begin{proof}
It directly follows from Theorem 4 of \cite{AES10}.
\end{proof}	
	
\begin{thm}
	Let $S$ be a scaling set corresponding to scaling function $\phi \in \mathfrak{D}_{N}^{M}$ where $M,N \geq 0$. If there are maximum $p^N$ many integers in $p^{M}S \cap [0,p^{M+N}-1]$, then, $\forall ~b \in \mathbb{Q}_p$,  
	\begin{equation*}
		\sum_{a \in I_p} \alpha_{a,b} \chi_{p}((a-b)\xi)=1, \hspace{.5cm} \forall ~\xi \in S.
	\end{equation*}
\end{thm}	

\begin{proof}
As there are atmost $p^N$ many integers in $p^{M}S \cap [0, p^{M+N}-1]$, using Theorem 5 of \cite{AES10}, $\phi(x-b)$, where $b \in \mathbb{Q}_p$, can be written as 
\begin{equation*}
	\phi(x-b) = \sum_{a \in I_p} \alpha_{a,b} ~ \phi(x-a)
\end{equation*}	
Taking Fourier transform in both side, we have
  \begin{equation} \label{d10}
  	\chi_{p}(b \xi) \hat{\phi}(\xi )= \sum_{a \in I_p} \alpha_{a,b} \chi_{p} (a \xi) \hat{\phi}(\xi)
  \end{equation}
  
 When $\xi \in S$, equation (\ref{d10}) reduces to $\chi_{p}(b \xi) = \sum \limits_{a \in I_p} \alpha_{a,b} \chi_{p} (a \xi)$.
  This completes the proof.
\end{proof}	
	

	

Scaling set satisfy following similar results to wavelet set.	
\begin{thm}
	Let $S \subset \mathbb{Q}_p$ be a set such that $\check{\mathds{1}}_S \in \mathfrak{D}_{N}^{M}$, where $M,N \geq 0$. If $S$ be a scaling set then $p^N$ many integers are in $p^{M}S \cap [0,p^{M+N}-1]$.
\end{thm}		

\begin{proof}
	As $S$ is a scaling set, $\{\check{\mathds{1}}_S (\cdot - a) : a \in I_p\}$ is a orthonormal set. Now using orthonormality and Parseval theorem
	\begin{eqnarray*}
		\delta_{a,0} 
		&=& \langle \check{\mathds{1}}_S (\cdot - a) , \check{\mathds{1}}_S \rangle \\ 
		&=& \langle \mathds{1}_S \chi_{p}(a \cdot) , \mathds{1}_S \rangle \\ 
		&=& \int \limits_{B_{M}(0)} \mathds{1}_S (\xi) \chi_{p}(a\xi) d\xi \\
		&=& \sum_{l=0}^{p^{M+N}-1} \int \limits_{B_{-N}\left(\frac{l}{p^M}\right)} \mathds{1}_S (\xi) \chi_{p}(a\xi) d\xi \\
		&=& \sum_{l=0}^{p^{M+N}-1} \mathds{1}_S \left(\frac{l}{p^M}\right) \chi_{p} \left( \frac{al}{p^M}\right) \int \limits_{B_{-N}\left(0\right)} \chi_{p}(a\xi) d\xi \\
		&=& \frac{1}{p^M} \mathds{1}_{\mathbb{Z}_p} (ap^N) \sum_{l=0}^{p^{M+N}-1} \mathds{1}_S \left(\frac{l}{p^M}\right) \chi_{p} \left( \frac{al}{p^M}\right)
		\end{eqnarray*}
Now $a=0$ give $\sum \limits_{l=0}^{p^{M+N}-1} \mathds{1}_S \left(\frac{l}{p^M}\right)=p^N$. This completes the proof.
\end{proof}	
	
\begin{thm}
	Let $S$ be a scaling set containing $0$. If $\check{\mathds{1}}_S$ is a test function then $S \subset \mathbb{Z}_p$.
\end{thm}
	
\begin{proof}
	This directly follows from Theorem 8 of \cite{AES10}.
\end{proof}

\begin{lem}
	Let $S$ be a scaling set corresponding $\phi$ and let\\ $\phi(\cdot -b) \in \bigcup \limits_{j \in \mathbb{Z}} V_j$ for any $b \in \mathbb{Q}_p$, where $V_j$ is defined by equation (\ref{vj}). Then $\bigcup \limits_{j \in \mathbb{Z}} p^{j}S =\mathbb{Q}_p.$ 
\end{lem}

\begin{proof}
From \cite{AES10}, we have $\bigcup \limits_{j \in \mathbb{Z}} \text{supp } \hat{\phi}(p^{j} \cdot) = \mathbb{Q}_p$.\\ Therefore, $\bigcup \limits_{j \in \mathbb{Z}} \text{supp } \mathds{1}_{p^{-j}S}= \bigcup \limits_{j \in \mathbb{Z}} p^{-j}S =\mathbb{Q}_p$. \\

This completes the proof.
\end{proof}

The concept of generalized scaling set in $L^{2}(\mathbb{R}^{n})$ was given by Brownik et al. \cite{BRS01}. We are giving here analogous definition of that.

\section{Generalized Scaling Set} 

In this section our intension is to discuss generalized scaling set and give its characterization.

\begin{defn}\label{gss}
	For a fixed $L \in \mathbb{N}$, a set $G \subset \mathbb{Q}_p$ is said to be a generalized scaling set of order $L$ if $\mu(G)=\frac{L}{p-1}$ and $p^{-1}G \setminus G$ is a multiwavelet set of order $L$.
\end{defn}

\begin{ex}
	$\mathbb{Z}_p$ is a generalized scaling set of order $p-1$ corresponding to Kozyrev's multiwavelet \cite{K02}.
\end{ex}

\begin{ex}
	$p^{-m+1}\mathbb{Z}_p$ be a generalized scaling set of order $(p-1)p^{m-1}$ for $m \in \mathbb{N}$ correspond to the multiwavelet given by Khrennikov and Shelkovich \cite{KS09}.
\end{ex}

In the following lemma we give a necessary and sufficient condition for a measurable set to be a generalized scaling set. One can observe similar result for $\mathbb{R}^n$ in \cite{BRS01}.

\begin{lem} \label{chargss}
	A measurable set $G \subset \mathbb{Q}_p$ is a generalized scaling set of order $L$ if and only if $G=\bigcup \limits_{j=1}^\infty p^{j}W$ a.e. for some multiwavelet set $W$ of order $L$.
\end{lem}

\begin{proof}
For some multiwavelet set $W$ of order $L$, let $G=\bigcup \limits_{j=1}^\infty p^{j}W$.\\
Clearly, $p^{-1}G \setminus G = W$. From \cite{HS19}, we know that $\mu(W)=L$. Now $$\mu(p^{-1}G \setminus G)= \mu(W) \Rightarrow \mu(G)= \frac{L}{p-1}.$$ So $G$ is a generalized scaling set of order $L$.\\

Conversely, suppose $G$ is a generalized scaling set of order $L$. Then $W:=p^{-1}G \setminus G$ is a multiwavelet set of order $L$.\\
As $\mu(W)=L$, $G \subset p^{-1}G$. Clearly, $\bigcup \limits_{j=1}^\infty p^{j}W \subset G$. Since both have same measure, $G=\bigcup \limits_{j=1}^\infty p^{j}W$ a.e.

\end{proof}

\begin{ex}
Consider $G=\mathbb{Z}_p$ and $W=\bigcup\limits_{n=1}^{p-1}(-\frac{n}{p}+\mathbb{Z}_p)$.\\ $$\bigcup\limits_{j=1}^{\infty}p^{j}W=pW \bigcup \bigcup \limits_{j=2}^{\infty}p^{j}W=\bigcup \limits_{n=1}^{p-1}(-n+p\mathbb{Z}_{p}) \bigcup \bigcup\limits_{j=2}^{\infty}p^{j}W.$$

Here $0 \notin pW$ and $\bigcup \limits_{j=2}^{\infty}p^{j}W \subset \mathbb{Z}_p$. Therefore, $G=\bigcup \limits_{j=1}^\infty p^{j}W$ a.e.	
\end{ex}

\begin{ex}
	Consider $G=p^{-m+1}\mathbb{Z}_p$ and $W=\bigcup\limits_{s \in J_{p,m}}(-s+\mathbb{Z}_p)$.\\
	 $$\bigcup\limits_{j=1}^{\infty}p^{j}W=\bigcup \limits_{s \in J_{p,m}}(-sp+p\mathbb{Z}_{p}) \bigcup \bigcup\limits_{j=2}^{\infty}p^{j}W.$$
	
	Here also $0 \notin pW$ and $\bigcup \limits_{j=2}^{\infty}p^{j}W \subset \mathbb{Z}_p$. Therefore, $G=\bigcup \limits_{j=1}^\infty p^{j}W$ a.e.	
\end{ex}

Now we will give a necessary condition for a generalized scaling set.

\begin{thm}
	Let $G \subset \mathbb{Q}_p$ be a generalized scaling set of order $L$. Then
	\begin{enumerate}[label=(\roman*)]
	\item $\mu(G)= \frac{L}{p-1}$.
	\item $G \subset p^{-1}G$.
	\item $\lim\limits_{n \rightarrow \infty} \mathds{1}_{G}(p^{n}\xi)=1$ for $\xi \in \mathbb{Q}_p$ a.e.	
	\end{enumerate}	
\end{thm}

\begin{proof}
	Let $G$ is a generalized scaling set of order $L$.\\ Then $(i)$ is guaranteed by Definition \ref{gss}.\\ $(ii)$ easily follows from Lemma \ref{chargss}.\\ Also from the last lemma we know that $G=\bigcup \limits_{j=1}^\infty p^{j}W$ for some multiwavelet set $W$ of order $L$. Furthermore, we know that, $\mu(p^{j}W \cap p^{l}W)=p^{-j}\delta_{j,l}L$. Therefore
	
	 $$\mathds{1}_{G}(p^{n}\xi) 
	 = \mathds{1}_{\bigcup \limits_{j=1}^{\infty}p^{j}W}(p^{n}\xi) =\sum \limits_{j=1}^{\infty}\mathds{1}_{p^{j}W}(p^{n}\xi) 
	 =\sum_{j=1}^{\infty} \mathds{1}_{W}(p^{n-j}\xi).$$
	As, $\sum \limits_{j \in \mathbb{Z}} \mathds{1}_{W}(p^{j}\xi)=1$ for $\xi$ a.e., we get condition $(iii)$.
\end{proof}	

\begin{rem}
We will conclude by raising following question. Like other field, is every scaling set is a generalized scaling set in this field?
\end{rem}

\section*{Acknowledgment}
The author is highly indebted to the fiscal support of Ministry of Human Resource Development (M.H.R.D.), Government of India. 



\end{document}